\documentclass[12pt]{amsart}    % This is LATEX.
\usepackage{charter}
\usepackage{fourier}
\usepackage{calrsfs}
\input xy
\xyoption{all}
\usepackage{amssymb}
\usepackage{amsfonts}
\usepackage{bm}

\usepackage{color}
\usepackage{epsfig}
\usepackage{graphicx}
\usepackage{graphics}
 \usepackage{fullpage}
\usepackage{comment}
\usepackage{hyperref}
\newfont{\vlte}{eufm10 at 22pt}
\newfont{\lte}{eufm10 at 18pt}
\newfont{\smb}{msbm6}
\newfont{\mmb}{msbm8}
\newfont{\tmb}{msbm10}
\newfont{\lmb}{msbm10 at 18pt}
\newcommand{\Gal}{\mbox{Gal}}

\newcounter{hw}
\newcommand{\be}{\begin{enumerate}}
\newcommand{\ee}{\end{enumerate}}

  \newcommand{\Aut}{\mbox{Aut}}

\newcommand{\calL}{{\mathcal L}}
\newcommand{\calM}{{\mathcal M}}

\newcommand{\F}{{\mathbb F}}

\newcommand{\N}{{\mathbb N}}
\newcommand{\Q}{{\mathbb Q}}
\newcommand{\R}{{\mathbb R}}
\newcommand{\Z}{{\mathbb Z}}

\newtheorem{theorem}{Theorem}[section]
\newtheorem{lemma}[theorem]{Lemma}

\newtheorem{proposition}[theorem]{Proposition}

\theoremstyle{definition}

\newtheorem{example}[theorem]{Example}

\theoremstyle{remark}
\newtheorem{remark}[theorem]{Remark}

\newtheorem{notation}[theorem]{Notation}

\newcommand{\isom}{\simeq}

\newcommand{\union}{\cup}

  \theoremstyle{plain}
\newtheorem{introprop}{Proposition}
\newtheorem{introthm}[introprop]{Theorem}

\begin{document}
\title{ON DECIDABLE ALGEBRAIC FIELDS}%
\author{Moshe Jarden}
\thanks{The first author was supported by the Minkowski Center for Geometry at Tel Aviv University,
established by the Minerva Foundation.}
\address{School of Mathematics\\ Tel Aviv University\\
Ramat Aviv, Tel Aviv 6139001, Israel}
\email{jarden@post.tau.ac.il}
\author{Alexandra Shlapentokh}%
\thanks{The second author has been partially supported by NSF grant DMS-1161456.}
\address{Department of Mathematics \\ East Carolina University \\ Greenville, NC 27858}%
\email{shlapentokha@ecu.edu }
\urladdr{www.personal.ecu.edu/shlapentokha} 

\date{\today}

\bigskip

\bigskip

\begin{abstract}
We prove the following propositions.
Theorem 1: Let $M$ be a subfield of a fixed algebraic closure $\tilde \Q$ of $\Q$ whose existential elementary theory is decidable
(resp. primitively decidable). Then, M is conjugate to a recursive (resp. primitive recursive) subfield $L \subset \tilde \Q$.

Theorem 2: For each positive integer $e$ there are infinitely many $e$-tuples $\boldsymbol \sigma \in \Gal(\Q)^e$
such that the field $\tilde \Q( {\boldsymbol \sigma})$ -- the fixed field of $\boldsymbol \sigma$,  is recursive in $\tilde\Q$ and its elementary theory is decidable. Moreover, $\tilde  \Q(\boldsymbol \sigma)$ is PAC and $\Gal(\tilde\Q(\boldsymbol \sigma))$ is isomorphic to the free profinite group on $e$ generators.
\end{abstract}%
\maketitle
\section{Introduction}

The main theme of this work is the interplay between
decidability of large algebraic extensions of $\Q$ and
their recursiveness in a fixed algebraic closure $\tilde \Q$
of $\Q$.
One of the main results of [JaK75] gives for each
positive integer $e$ a recursive procedure to
decide whether a sentence $\theta$ in the language of rings is
true in the field $\tilde \Q({\boldsymbol \sigma})$ for all
${\boldsymbol \sigma}\in\Gal(\Q)^e$ outside a set of Haar measure zero
(see also [FrJ08, p.~442, Thm.~20.6.7]).
Here, 
$\Gal(\Q)=\Gal(\tilde \Q/\Q)$ is the absolute Galois group
of $\Q$, and for each
${\boldsymbol \sigma}=(\sigma_1,\ldots,\sigma_e)\in\Gal(\Q)^e$, \
$\tilde \Q({\boldsymbol \sigma})$ is the fixed field of
$\sigma_1,\ldots,\sigma_e$ in $\tilde \Q$.
The results of [FHJ84] even 
give a primitive recursive procedure for
the same decision problem (see also [FrJ08, p.~722, Thm.~30.6.1]).

Note that the above procedures give no information about
individual fields of the form $\tilde \Q({\boldsymbol \sigma})$.
Indeed, by Proposition \ref{APPa}, there
are uncountably many elementary equivalence classes of
fields $\tilde \Q({\boldsymbol \sigma})$.
On the other hand, since the language of rings is countable,
there are at most countably many decision procedures.
Hence, all but at most countably many fields of the form
$\tilde \Q({\boldsymbol \sigma})$ are undecidable.

Another question that one may ask in this context is about
the relation between an individual field $\tilde \Q({\boldsymbol \sigma})$
and $\tilde \Q$.
To this end we recall that one may order the elements of
$\tilde \Q$ in a primitive recursive sequence and give a
primitive recursive procedure to carry out the field
theoretic operations among the elements of that sequence.
It therefore makes sense to ask about a subfield $M$ of
$\tilde \Q$ whether $M$ is a recursive subset of $\tilde \Q$
(in which case $M$ is also a recursive (or a computable)
subfield of $\tilde \Q$).

Usually, this is not the case, because $\tilde \Q$ has only
countably many recursive subsets.
Even if the elementary theory of $M$ is decidable, it may
happen that $M$ has uncountably many conjugates (Example
\ref{RP}, 
when $M$ is a real or a $p$-adic closure of $\Q$).
The elementary theory of each of them is the same as that of
$M$, so is also decidable.
But only countably many of them are recursive in $\tilde \Q$. 

We shed light on these problems by proving two results:

\begin{introthm}
\label{A}
Let $M$ be a subfield of\/ $\tilde \Q$
whose existential
elementary theory is decidable (resp. primitively decidable).
Then, $M$ is conjugate to a recursive
(resp. primitive recursive) subfield $L$ of $\tilde \Q$.
\end{introthm}

In view of this theorem,
given a subfield of $\tilde \Q$ with 
undecidable existential (or elementary) theory in the 
language of rings, one can distinguish between two cases. 
The theory can be undecidable because the field has no 
computable conjugate within the given copy of
$\tilde \Q$ or the theory can be undecidable for a
different arithmetic reason. 
In the first case it is tempting to say that the theory is 
{\it trivially} undecidable.  
A simple example of  a field with a trivially 
undecidable existential theory is a Galois extension of $\Q$ 
which is not recursive as a subset of $\tilde \Q$.  

\begin{introthm}
\label{B}
For each positive integer $e$ there are infinitely many
$e$-tuples ${\boldsymbol \sigma}\in\Gal(\Q)^e$ such that the field
$\tilde \Q({\boldsymbol \sigma})$ is recursive in $\tilde \Q$ and its
elementary theory is decidable.
Moreover, $\tilde \Q({\boldsymbol \sigma})$ is PAC and
$\Gal(\tilde \Q({\boldsymbol \sigma}))$ is isomorphic to the free profinite
group on $e$ generators.
\end{introthm}

Both theorems make sense, because we can list the elements
of $\tilde \Q$ in a primitive recursive sequence.
The proof of Theorem \ref{A} depends on our ability to perform
the basic field theoretic operations including the
factorization of
polynomials over given number fields and even over $\tilde \Q$
in a primitive recursive way.
The proof of Theorem \ref{B} uses in addition a recursive (but not
primitive recursive) version of Hilbert irreducibility
theorem (Lemma \ref{DECa}).

All of these operations can be carried out over 
each given finitely generated field (over its prime field).
In the terminology of [FrJ08, Chap.~19], these fields have
the ``elimination theory''.
So, actually we prove Theorems \ref{A} and \ref{B} for fields with
elimination theory.
In particular, they hold for finitely generated fields.

Theorem 13.3.5 and Proposition 13.2.1 of [FrJ08] state that
every infinite finitely generated field $K$ is Hilbertian.
An analysis of their proofs seems to show that the procedure
to find a point in a given Hilbertian subset of $K^r$ is
primitive recursive.
If this is true, we may strengthen Theorem B and replace
``recursive'' and ``decidable'' in that theorem by
``primitive recursive'' and ``primitively decidable'',
respectively.
However, carrying out this check in the present work will
take us away from its main topic. 
So, we don't do it here.\\

{\large Acknowledgement}:
The authors are indebted to Aharon Razon for critically
reading a draft of this work.

%End of file intr.tex
%File recursive.tex, 10 October 2014
%12 October 2014
%28 November 2014
%30 December 2014
%29 January 2015

\section{Recursive Subfields of $\tilde K$}
\label{RECR}

We consider a presented field $K$ with {\bf elimination
theory} in the sense of [FrJ08, p.~410, Def.~19.2.8].
This is a field which is explicitly constructed from the
ring $\Z$ of integers,
one has ``effective recipes'' to add and multiply given elements
and to ``effectively compute'' the inverse of each given non-zero
element.
In particular, $K$ is countable.
Most important, there is an ``effective algorithm'' to
factor each given non-zero polynomial $f\in K[Z]$ into
a product of irreducible polynomials.
Moreover, it is possible to ``effectively adjoin'' a root
$z$ of $f$ to $K$.
Each element $z'$ of $K(z)$ is then uniquely given as a sum
$\sum_{i=0}^{d-1}a_iz^i$ with $d=\deg(f)$ and
$a_0,\ldots,a_{d-1}\in K$,
and one can effectively compute $\mbox{irr}(z',K)$.
The field $K(z)$ has again the elimination theory.
An effective version of the primitive element theorem is
also true, that is if $z_1,\ldots,z_n$ are roots of given
irreducible separable polynomials $f_1,\ldots,f_n\in K[Z]$,
respectively, then one
can ``effectively compute'' an irreducible separable
polynomial $f\in K[Z]$ and a root $z$ of $K$ such that
$K(z_1,\ldots,z_n)=K(z)$.
Moreover, one can ``effectively present'' $z$ as a linear
combination of $z_1,\ldots,z_n$ with coefficients in $K$ and
``effectively present'' each $z_i$ as a polynomial in $z$
with coefficients in $K$.

All of these notions and algorithms are rigorously defined,
explained, and proved in [FrJ08, Sections 19.1 and 19.2].
Moreover, it is proved there that the above algorithms are
{\bf primitive recursive} in the usual sense (e.g.~as defined
in [FrJ08, Sec.~8.4]).
In this case we also say that the above algorithms are
{\bf effective}.
It is further proved in [FrJ08, Sec.~19.4] that both the
separable closure $K_s$ and the algebraic closure
$\tilde K$ of $K$
can be presented, and then have elimination theory.

Having done so, we say that a subfield $M$ of $\tilde K$
is {\bf recursive}  (resp.~{\bf primitive
recursive}), if $M$ is a recursive (resp.~primitive
recursive) subset of $\tilde K$ (e.g.~[FrJ08, Sections 8.4 and
8.5, where in each case one examines the characteristic
function of the subset]).
Since addition, multiplication, and taking inverse of
elements of $\tilde K$ are primitive recursive, these
operations in $M$ are also recursive (resp.~primitive recursive).

Independently of the question whether $M$ is a recursive
subfield of $\tilde K$ or not, we may consider the language
$\calL(\mbox{ring},K)$ of rings with a constant symbol for each
element of $K$.
An {\bf existential sentence} in $\calL(\mbox{ring},K)$ is a
sentence which is equivalent to a sentence of the form
$(\exists X_1)\cdots(\exists X_n)
[\bigvee_i\bigwedge_jf_{ij}(X_1,\ldots,X_n)=0]$
[FrJ08, p.~462].
Let $\mbox{Ex}(M)$ be the set of all existential sentences in
$\calL(\mbox{ring},K)$ which are true in $M$.
This is a subset of the {\bf elementary theory} $\mbox{Th}(M)$
consisting of all sentences of $\calL(\mbox{ring},K)$ that hold in
$M$.
We say that $\mbox{Ex}(M)$ is {\bf decidable}
(resp.~{\bf primitively decidable}) if there exists an
algorithm (resp.~primitive recursive algorithm) to decide
whether a given existential sentence of $\calL(\mbox{ring},K)$
holds in $M$ or not.

The procedures we describe below use verbs like
``construct'', ``find'', ``compute'', etc.
When these terms are preceded by the adverb ``effectively'',
then the corresponding
parts of the procedures are primitive recursive.

By definition, each primitive recursive subset of $\tilde K$
is recursive.
Similarly, if $\mbox{Th}(M)$ (resp.~$\mbox{Ex}(M)$) is primitively decidable,
then $\mbox{Th}(M)$ (resp.~$\mbox{Ex}(M)$) is decidable.
(See also a comparison in [FrJ08, pp.,159-150, Sec.~8.6] between
recursive and primitive recursive decidability procedures.)

Note that, in the cases we consider,
$\mbox{Th}(M)$ involves only constant symbols of $K$ but not of
$\tilde K\setminus K$.
Thus, the question whether $M$ is a recursive
(respectively, primitive recursive) subfield of
$\tilde K$ is independent of the question whether $\mbox{Th}(M)$ is
decidable (resp.~primitively recursiv).
Indeed, even if $\mbox{Th}(M)$ is decidable, $M$ may have
uncountably many $K$-conjugates in $\tilde K$ (Example \ref{RP}).
Since $K$ is countable, so is the language $\calL(\mbox{ring},K)$.
Hence, $\tilde K$ has only countably many recursive subfields.
It follow that at most countably many of the $K$-conjugates
of $M$ in $\tilde K$ are recursive in $\tilde K$.
All the others are non-recursive in $\tilde K$.

However, we prove in this section that if $M$ is a field
extension of $K$ in $\tilde K$ and $\mbox{Ex}(M)$ is decidable
(resp.~recursively decidable), then
$M$ has a $K$-conjugate $L$ which is recursive
(resp.~primitive recursive) in $\tilde K$.
In particular, $\mbox{Th}(L)=\mbox{Th}(M)$, hence $\mbox{Ex}(L)=\mbox{Ex}(M)$,
so $\mbox{Ex}(L)$ is in addition decidable (resp.~primitively decidable).

\begin{remark}
\label{ROOT}
Let $K$ be a field,
$p\in K[Z]$, and $\phi\colon L\to L'$ a $K$-isomorphism
of subfields of $\tilde K$.
Denote the set of roots of $p$ in $\tilde K$ by $P$.
Then, $\phi(P\cap L)\subset P\cap L'$
and $\phi^{-1}(P\cap L')\subset P\cap L$.
Therefore, $\phi(P\cap L)=P\cap L'$.
\end{remark}

\begin{lemma}
\label{POLY}
Let $M$ be a subfield of $\tilde K$ that contains $K$
such that $\mbox{Ex}(M)$ is decidable
(resp.~recursively decidable).
Suppose we are given
\begin{itemize}
\item[(a1)]
a finite separable extension $L$ of
$K$ and a $K$-embedding of $L$ into $M$, and
\item[(a2)]
a monic separable polynomial $p$ in $K[Z]$.
\end{itemize}
Let $P$ be the set of roots of $p$ in $K_s$.
Then,
\begin{itemize}
\item[(b)]
we can determine (resp.~effectively determine)
whether there exists a $K$-embedding of $L(z)$ into $M$
with $z\in P\setminus L$;
\end{itemize}
\end{lemma}

\begin{proof}
We effectively decompose $p(Z)$
into a product of monic irreducible factors over $L$,
$$
p(Z)=(Z-a_1)\cdots(Z-a_l)h_1(Z)\cdots h_m(Z).
$$
such that $a_1,\ldots,a_l\in L$
and $\deg(h_i)\ge2$ for $i=1,\ldots,m$
[FrJ08, p.~407, Lemma 19.2.2].
If $m=0$, then $P\subset L$,
so there is no embedding of
$L(z)$ into $M$ with $z\in P\setminus L$.

Otherwise, we effectively construct a primitive element $y$ for $L/K$,
effectively compute $f=\mbox{irr}(y,K)$, and set $d=\deg(f)$.
For each $1\le i\le m$ we set $d_i=\deg(h_i)$.
Then, we effectively compute for each $0\le j\le d_i$ the unique
polynomial $g_{ij}$ in $K[Y]$ of degree at most $d-1$ such
that
$h_i(Z)=\sum_{j=0}^{d_i}g_{ij}(y)Z^j$.
We set $g_i(Y,Z)=\sum_{j=0}^{d_i}g_{ij}(Y)Z^j$
and observe that $g_i\in K[Y,Z]$ and $g_i(y,Z)=h_i(Z)$.
Then, we denote the existential sentence
$$
(\exists Y)(\exists Z)[f(Y)=0\land g_i(Y,Z)=0]
$$
of $\calL(\mbox{ring},K)$ by $\theta_i$.

Since the existential theory of $M$ in the language
$\calL(\mbox{ring},K)$
is decidable (resp.~pri\-mitively decidable),
we may check (resp.~effectively check)
the truth of each $\theta_i$ in $M$.
If none of the sentences $\theta_1,\ldots,\theta_m$ is true in
$M$, then there exists no $K$-embedding $\phi'$ of
$L(z)$ into $M$ with $z\in P\setminus L$.

Indeed, if such $z$ and $\phi'$ exist, we write $y'=\phi'(y)$ and
$z'=\phi'(z)$.
Then, $z\in P\setminus\{a_1,\ldots,a_l\}$,
so there exists $1\le i\le m$ with $h_i(z)=0$.
Hence, $g_i(y,z)=0$.
Applying $\phi'$, we see that $f(y')=0$ and $g_i(y',z')=0$,
with $y',z'\in M$.
Thus, $\theta_i$ holds in $M$, in contrast to our assumption.

Finally suppose that one of the sentences
$\theta_1,\ldots,\theta_m$, say $\theta_1$, is true in $M$.
Thus, there exist
$y',z'\in M$ with $f(y')=0$ and
$g_1(y',z')=0$.
Since $f$ is irreducible over $K$,
we may extend (resp.~effectively extend)
the map $y\twoheadrightarrow y'$ to a $K$-isomorphism
$\phi'_1$ of $L=K(y)$ onto $K(y')$.
Since $g_i(y,Z)=h_i(Z)$ is irreducible over $K(y)=L$,
the polynomial $g_i(y',Z)$ is irreducible over $K(y')$.
Since $z'$ is a root of the latter polynomial,
we may find (resp.~effectively find) a root $z$ of $g_i(y,Z)$
in $K_s$ and conclude that the isomorphism
$(\phi'_1)^{-1}\colon K(y')\to K(y)$
extends to an isomorphism $K(y',z')\to K(y,z)$ that maps
$z'$ onto $z$.
In particular, $h_1(z)=g_1(y,z)=0$, so $z\in P\setminus L$.
Then, the inverse isomorphism $\phi'\colon K(y,z)\to K(y',z')$
of the latter isomorphism is the desired one.
\end{proof}

\begin{remark}
\label{NEXT}
Note that the $K$-embedding $\phi'\colon L(z)\to M$
constructed in the last paragraph of the latter proof does
not necessarily extend the $K$-embedding $\phi\colon L\to M$.
\end{remark}

\begin{lemma}
\label{IND}
Let $M$ be a subfield of $\tilde K$ that contains $K$
such that $\mbox{Ex}(M)$ is decidable (resp.~recursively decidable).
Suppose we are given
\begin{itemize}
\item[(a1)]
a finite separable extension $L$ of
$K$ and a $K$-embedding $\phi\colon L\to M$ and
\item[(a2)]
a monic separable polynomial $p$ in $K[Z]$.
\end{itemize}
Let $P$ be the set of roots of $p$ in $K_s$.
Then, we can find (resp.~effectively find) a subset $I$ of $P$
such that there exists a
$K$-embedding $\psi\colon L(I)\to M$ with the property that
$I=P\cap L(I)$ and $\psi(I)=P\cap M$.
\end{lemma}

\begin{proof}
The assumptions of our lemma coincide with the assumptions of
Lemma \ref{POLY}, so the conclusion of that
lemma holds in our situation.
We set $L'=\phi(L)$ and consider the two possible cases.
\begin{itemize}
\item[\bf Case A:]
{There is no $K$-embedding $L(z)\to M$
with $z\in P\setminus L$}
Then, we set $I=P\cap L$ and $\psi=\phi$.
Hence, $L(I)=L$, so $I=P\cap L(I)$.
By Remark \ref{ROOT}, $\psi(P\cap L)=P\cap L'$.

Note that $\psi(I)=\psi(P\cap L)\subset P\cap M$.
If there exists $z'\in P\cap M\setminus\psi(I)$,
then by the preceding paragraph, $z'\in M\setminus L'$.
Thus, there exists $z\in K_s$ and an extension of $\psi$ to
an isomorphism $\psi'\colon L(z)\to L'(z')$ such that
$\psi'(z)=z'$.
Hence, $z\in P\setminus L$, in contrast to our assumption.

It follows from this contradiction that $\psi(I)=P\cap M$.

\item[\bf Case B:]
{Case A does not occur}
Using Lemma \ref{POLY},
we find (resp.~effectively find) $z\in P\setminus L$ and construct
(effectively construct) a
$K$-embedding $\phi'$ of $L(z)$ into $M$.
Then, $|P\setminus L(z)|<|P\setminus L|$.
By induction, we may find (resp.~effectively find)
a subset $I$ of $P$ and a
$K$-embedding $\psi$ of $L(z,I)$ into $M$ such that $I=P\cap
L(z,I)$ and $\psi(I)=P\cap M$.
In particular, since $z\in P$, we have $z\in I$.
Thus, $L(I,z)=L(I)$ and $I=P\cap L(I)$.
\end{itemize}
\end{proof}

\begin{lemma}
 \label{RES}
Let $M$ be a subfield of $\tilde K$ that contains $K$,
let $p\in K[Z]$ be a monic separable polynomial,
and let $P$ be the set of all roots of $p$ in $K_s$.
Consider an extension $L$ of $K$ in $K_s$
and let $\phi$ and $\phi'$ be $K$-embeddings of $L$ into $M$.
Suppose that $\phi(P\cap L)=P\cap M$.
Then, $\phi'(P\cap L)=P\cap M$.
\end{lemma}

\begin{proof}
We set $M_0=\phi(L)$ and $M'_0=\phi'(L)$
and let $\tau\colon M_0\to M'_0$ be the $K$-isomorphism that
satisfies $\tau\circ\phi=\phi'$.
By Observation \ref{ROOT},
$\phi(P\cap L)=P\cap M_0$,
$\phi'(P\cap L)=P\cap M'_0$,
and $\tau(P\cap M_0)=P\cap M'_0$.
It follows from $\phi(P\cap L)=P\cap M$ that
$P\cap M=P\cap M_0$.
By the relation $\tau(P\cap M_0)=P\cap M'_0$ and the
injectivity of $\tau$, we have
$|P\cap M'_0|=|P\cap M_0|=|P\cap M|$.
Since $P\cap M'_0\subset P\cap M$,
we deduce that $P\cap M'_0=P\cap M$.
Finally,
$\phi'(P\cap L)
=\tau(\phi(P\cap L))
=\tau(P\cap M_0)
=P\cap M'_0
=P\cap M$,
as asserted.
\end{proof}

We denote the maximal purely inseparable extension of a field
$F$ by $F_{\mbox{ins}}$.

\begin{lemma}
\label{INS}
Let $F$ be a recursive (resp.~primitive recursive)
subfield of $K_s$ that contains $K$.
Suppose that we can decide (resp.~primitively
decide) for each monic separable polynomial $f\in
K[X]$ whether $f$ has a root in $F$.
Then, $F_{\mbox{ins}}$ is also a recursive (resp.~primitive
recursive) subfield of $\tilde K$ and
we can decide (resp.~primitively decide)
for each monic polynomial $f\in K[X]$ whether $f$ has a root in
$F_{\mbox{ins}}$.
\end{lemma}

\begin{proof}
It suffices to consider the case where $p=\mbox{chr}(K)>0$.
By our assumptions on $K$,
we are able to effectively factor each monic $f\in K[X]$
into irreducible polynomials over $K$.
Hence, in order to decide whether $f$ has a root in
$F_{\mbox{ins}}$, we may assume that $f$ is irreducible.
In this case we write $f(X)=g(X^q)$,
where $g$ is an irreducible separable polynomial in $K[X]$
and $q$ is a power of $p$.
If $y$ is a root of $g$ in $F$ and $z^q=y$ with $z\in\tilde K$,
then $f(z)=0$ and $z\in F_{\mbox{ins}}$.
On the other hand, if $z\in F_{\mbox{ins}}$ and $f(z)=0$,
we get for $y=z^q$ that $g(y)=f(z)=0$ and
$y\in K_s\cap F_{\mbox{ins}}=F$.
By assumption, we may decide (resp.~effectively decide)
whether $g$ has a root in $F$.
Hence, we may decide (resp.~effectively decide)
whether $f$ has a root in $F_{\mbox{ins}}$.

Consider $x\in\tilde K$ and compute $f(X)=\mbox{irr}(x,K)$.
Then, we write $f(X)=g(X^q)$,
where $g\in K[X]$ is separable
and $q$ is a power of $p$.
Then, $x^q\in K_s$ and we can check
(resp.~effectively check) whether $x^q\in F$,
because $F$ is a recursive (resp.~primitive recursive)
subfield of $K_s$.
This is the case if and only if $x\in F_{\mbox{ins}}$.
\end{proof}

\begin{theorem}
\label{SHL}
Let $M$ be a perfect subfield of $\tilde K$ that
contains $K$.
Suppose that the existential theory of $M$ in
$\calL(\mbox{ring},K)$ is
decidable (resp.~primitively decidable).
Then, $\tilde K$ has a recursive
(resp.~primitive recursive) subfield $L$ that contains $K$
and is $K$-isomorphic to $M$.
\end{theorem}

\begin{proof}
We make a primitive recursive list, $p_1,p_2,p_3,\ldots$,
of all monic separable irreducible polynomials in $K[Z]$.
For each positive integer $j$ we compute the set $P_j$ of
all roots of $p_j$ in $K_s$.
Since $p_1,p_2,p_3,\ldots$ are distinct and irreducible
polynomials,
\begin{enumerate}
\item \label{cf1}
the sets $P_1,P_2,P_3,\ldots$ are disjoint.

Using Lemma \ref{IND}, we inductively construct a
recursive (resp.~primitive recursive)
ascending tower $L_0\subset L_1\subset L_2\subset\cdots$
of finite extensions of $K$ in $K_s$ with $L_0=K$.
Moreover, for each positive integer $j$ there is a $K$-embedding
$\psi_j\colon L_j\to M$ such that
\item \label{cf2}
$\psi_j(P_j\cap L_j)=P_j\cap M$.

We consider the subfield
$L_\infty=\union_{j=1}^\infty L_j$ of $K_s$ (hence, also of
$\tilde K$) and set $M_\infty=M\cap K_s$.

For each positive integer $j$ we let $E_j$ be the set of all
$K$-embeddings $\psi\colon L_j\to M$ such that
$\psi(P_j\cap L_j)=P_j\cap M$.
The set $E_j$ is non-empty (because $\psi_j\in E_j$) and
finite (because $[L_j:K]<\infty$).
Moreover, if $\phi\in E_{j+1}$, then by Lemma \ref{RES},
$\phi|_{L_j}\in E_j$.
Thus, $E_0,E_1,E_2,\ldots$ form an inverse system of
finite sets.
Hence, by [FrJ08, p.~3, Cor.~1.1.4],
there exists a $K$-embedding
$\phi\colon L_\infty\to M$
such that $\phi|_{L_j}\in E_j$, i.e.
\item \label{cf3}
$\phi(P_j\cap L_j)=P_j\cap M$

\noindent
for each positive integer $j$.
\end{enumerate}
 {\bf Claim A}: {$\phi$ maps $L_\infty$ isomorphically onto $M_\infty$}
Indeed, since $\phi(L_\infty)\subset M$ and $L_\infty\subset K_s$, we have
$\phi(L_\infty)\subset M_\infty$.
Conversely, let $y\in M_\infty$.
By definition, $K_s=\union_{j=1}^\infty P_j$.
Hence, there exists a positive integer $j$ such that
$y\in P_j$.
By \eqref{cf3} there exists $x\in P_j\cap L_j$ (hence, $x\in
L_\infty$) such that $\phi(x)=y$.
Thus, $\phi(L_\infty)=M_\infty$.
Since $\phi$ is injective, $\phi$ maps $L_\infty$
isomorphically onto $M_\infty$.\\
 {\bf Claim B}:
{$L_\infty=\union_{j=1}^\infty P_j\cap L_j$}
Indeed, by definition, the right hand side is contained in
the left hand side.
Conversely, let $x\in L_\infty$.
Then, there exists a positive integer $j$ with $x\in P_j$.
Thus, $\phi(x)\in P_j\cap M$.
By \eqref{cf3}, $x\in P_j\cap L_j$, as claimed.\\
 {\bf Claim C}: The field $L_\infty$ is recursive (resp.~primitive
recursive) in $K_s$.
Indeed, given $x\in K_s$ we can effectively find a positive integer $j$
with $p_j(x)=0$.
This means that $x\in P_j$.
Then, we check (resp.~effectively check) if $x\in L_j$.
If this is the case, then $x\in L_\infty$.
Otherwise, $x\notin L_\infty$.
Indeed, if $x\in L_\infty$, then $\phi(x)\in
M_\infty=\union_{j'=1}^\infty P_{j'}\cap M_\infty$.
Thus, there exists a positive integer $j'$ with
$\phi(x)\in P_{j'}$.
Since $\phi(P_j)=P_j$, it follows from \eqref{cf1} that $j'=j$,
Therefore, by \eqref{cf3}, $x\in L_j$,
in contrast to our assumption.\\
{\bf Conclusion of the proof}:  Since $M$ is perfect and $M_\infty=M\cap K_s$, the field $M$
is the maximal purely inseparable extension of $M_\infty$.
Let $L$ be the maximal purely inseparable extension of
$L_\infty$
and extend $\phi$, using Claim A, in the unique possible way to an
isomorphism $\phi\colon L\to M$.
By Claim C and Lemma \ref{INS}, $L$ is a recursive
(resp.~primitive recursive) subfield of $\tilde K$.
\end{proof}

\begin{remark}
\label{ISM}
Note that we do not claim nor we do not prove that the
$K$-isomorphism $L\to M$ mentioned in Theorem \ref{SHL}
is recursive.
Indeed, let $\calM$ be a $K$-conjugacy class of fields with
an existential decidable theory.
Only countably many of them are recursive subfields of $\tilde K$.
For each of them there are only countably many recursive
$K$-embeddings into $\tilde K$.
Thus, all but countably many fields in $\calM$ are not
images of those embeddings.
\end{remark}

\begin{example}
\label{RP}
Let $R$ be a real closure of $\Q$.
Then, $R$ is elementarily equivalent to the field $\R$ of
real numbers [Pre81, p.~51, Cor.~5.3].
By Tarski [Tar48, p.~42, Thm.~37],
$\mbox{Th}(\R)$ is primitively decidable, hence so are $\mbox{Th}(R)$
and $\mbox{Ex}(R)$.
Similarly, for each positive integer $p$, we choose a
Henselization $\Q_p$ of $\Q$ with respect to the $p$-adic
valuation of $\Q$.
Then, $\Q_p$ is elementary equivalent to the field
$\hat \Q_p$ of all $p$-adic numbers [PrR84, p.~86, Thm.~5.1].
We know that
$\mbox{Th}(\hat \Q_p)$ is decidable
[Mar02, p.~97, Cor.~3.3.16],
and even primitively decidable [Wei84, p.~84, Cor.~3.11(iii)].
Hence, so are $\mbox{Th}(\Q_p)$ and $\mbox{Ex}(\Q_p)$.

By Emil Artin, $\Aut(R)$ is trivial [Lan93, p.~455, Thm.~XI.2.9].
By F.~K.~Schmidt [Sch33], the same is true for $\Aut(\Q_p)$
[Jar91, Prop.~14.5].
Since $\Gal(\Q)$ is uncountable, the fields $R$ and
$\Q_p$ have uncountably many $\Q$-conjugates.
It follows from Theorem \ref{SHL}
that there exists a primitive recursive
subfield $L$ of $\tilde \Q$ which is isomorphic to $R$
(resp.~to $\Q_p$).
However, by Remark \ref{ISM}, there exists a conjugate
$R'$ of $R$ (resp.~$\Q'_p$ of $\Q_p$)
which is not the image of a recursive
subfield $L$ of $\tilde \Q$ by a recursive isomorphism.
\end{example}

%End of file recr.tex
%File decide.tex, 3 October 2014
%9 October 2014
%30 November 2014
%11 December 2014
%29 January 2015

\section{Decidable Large Fields}
\label{DEC}

We consider a presented field $K$ with elimination theory
as in Section \ref{RECR}.
In addition to the field operations discussed in that
section, we note that all of the standard operations on
Galois extensions of $K$ and of given finite extensions of
$K$ and with their Galois groups can be
carried out in a primitive recursive way
[FrJ08, pp.~411--412, Sec.~19.3].
In addition,
using [FrJ08, p.~413, Lemma 19.4.1], we effectively
construct a sequence $z_1,z_2,z_3,\ldots$
of elements that presents $K_s$ over $K$.
We can also effectively construct a sequence
$\tilde z_1,\tilde z_2,\tilde z_3,\ldots$ of elements that presents
$\tilde K$ over $K$.
Thus, $K_s=K(z_1,z_2,z_3,\ldots)$ and
$\tilde K=K(\tilde z_1,\tilde z_2,\tilde z_3,\ldots)$.
Again, we recall that by the above mentioned lemma,
both fields have the splitting algorithm.

\begin{notation}
\label{HIL}
We consider variables $T_1,\ldots,T_r,X$ over $K$ and
abbreviate $T_1,\ldots,T_r$ by $\mathbf T$.  Let $f_1,\ldots,f_m$ be polynomials in $K(\mathbf T)[X]$
which are irreducible and separable in the ring $K(\mathbf T)[X]$
and let $g$ be a non-zero
polynomial in $K[\mathbf T]$.
Following [FrJ08, Sec.~12.1], we write
$H_K(f_1,\ldots,f_m;g)$ for the set of all $\mathbf a\in K^r$
such that $f_1(\mathbf a,X),\ldots,f_m(\mathbf a,X)$ are defined,
irreducible, and separable in $K[X]$.
In addition, $g(\mathbf a)$ is defined and non-zero.
Then, we call $H_K(f_1,\ldots,f_r;g)$ a {\bf separable
Hilbert subset} of $K^r$.
A {\bf separable Hilbert set} of $K$ is a separable Hilbert
subset of $K^r$ for some positive integer $r$. We say that
$K$ is {\bf Hilbertian} if each separable Hilbert set of $K$
is non-empty.
%If $H_K(f_1,\ldots,f_m)$ is non-empty for all $m$-tuples
%$f_1,\ldots,f_m$ as above, then $K$ is Hilbertian
%[FrJ08, p.~322, Lemma 12.1.6].
\end{notation}

\begin{lemma}
 \label{DECa}
Suppose $K$ is Hilbertian.
Given a separable Hilbert subset $H$ of $K^r$, we can
recursively find
$(a_1,\ldots,a_r)\in H$.
\end{lemma}

\begin{proof}
Using Cantor's first diagonal method, we can effectively
write down a list $(\mathbf a_1,\mathbf a_2,\mathbf a_3,\ldots)$,
with $\mathbf a_i=(a_{i1},\ldots,a_{ir})$,
of all elements of $K^r$.
Let $H=H_K(f_1,\ldots,f_m;g)$ as in Notation \ref{HIL}.

Since $K$ has the splitting algorithm, we may effectively check the
irreducibility of the polynomials
$f_1(\mathbf a_i,X),\ldots,f_m(\mathbf a_i,X)$ over $K$ and their
separability for $i=1,2,3,\ldots$, and also the condition
$g(\mathbf a_i)\ne0$.
Since $K$ is Hilbertian, we will certainly hit an $i$  such that
$f_1(\mathbf a_i,X),\ldots,\mathbf f_m(\mathbf a_i,X)$ are defined,
irreducible, and separable over $K$, and $g(\mathbf a_i)\ne0$,
as needed.

This gives us a recursive procedure
(but not a primitive recursive procedure) to find $\mathbf a$ in $H$.
\end{proof}

\begin{lemma}
\label{DECb}
Suppose that $K$ is Hilbertian, $L$ is a given finite
separable extension of $K$, and $H$ is a given separable Hilbert
subset of $L^r$.
Then, we can effectively find
a separable Hilbert subset $H_K$ of $K^r$ which is contained in $H$.
\end{lemma}

\begin{proof}
Let $H=H_L(f_1,\ldots,f_m;g)$, where $f_1,\ldots,f_m$ are
irreducible separable polynomials in $L(\mathbf T)[X]$ and $g\in
L[\mathbf T]$ non-zero.
Without loss we may assume that the coefficients of
$f_1,\ldots,f_m$ are in $L[\mathbf T]$.
The proof of [FrJ08, p.~224, Lemma 12.2.2] uses the proof of
[FrJ08, p.~223, Lemma 12.2.1]
with $L(T_1,\ldots,T_r)$ replacing $L$ in the
latter lemma to effectively produce
a non-zero polynomial $h\in L[T_1,\ldots,T_r]$
and irreducible separable polynomials
$p_1,\ldots,p_m$ in $K(T_1,\ldots,T_r)[X]$
with the following property:
If $\mathbf a\in K^r$, $h(\mathbf a)\ne0$,
and the $p_i(\mathbf a,X)$'s are defined, irreducible, and separable in
$K[X]$,
then $f_1(\mathbf a,X),\ldots,f_m(\mathbf a,X)$ are defined,
irreducible, and separable in $L[X]$, and $g(\mathbf a)\ne0$.
Replacing $h$ by the product of all its $K$-conjugates
(an effective operation),
$H_K(p_1,\ldots,p_m;h)$ is a separable Hilbert subset of $K^r$
which is contained in $H$.
\end{proof}

Recall that a field $M$ is {\bf PAC} if every 
absolutely integral variety over $M$ has an $M$-rational
point.
We denote the free profinite group on $e$ generators by
$\\hat F_e$ [FrJ08, p.~349, first paragraph].
We also denote the absolute Galois group of a field $F$ by
$\Gal(F)$.

\begin{lemma}
\label{JAX}
Let $M$ be an extension of $K$ in $\tilde K$.
Suppose $M$ is perfect and PAC, $\Gal(M)\isom \hat F_e$, and one may check
whether a given monic polynomial in $K[X]$ has a root in
$M$.
Then, $\mbox{Th}(M)$ is decidable.
\end{lemma}

\begin{proof}
Let $\mbox{Root}(M)$ be the recursive set of all monic polynomials
in $K[X]$ that have a root in $M$.
Let $\mbox{Ax}(K,e)$ be the set of axioms in the language
$\calL(\mbox{ring},K)$ given in [FrJ08, p.~437, Prop.~20.4.4].
Thus, a field extension $F$ of $K$
satisfies $\mbox{Ax}(K,e)$ if and only if $F$ is perfect, PAC, and
$\Gal(F)\isom \hat F_e$.
Let $\mbox{Ax}(K,M)$ be the union of $\mbox{Ax}(K,e)$ and the axioms
that say that each $f\in\mbox{Root}(M)$ has a root in $M$ and each
monic $f\in K[X]\setminus\mbox{Root}(M)$ does not have a root in
$M$.
In particular, $M\models\mbox{Ax}(K,M)$.
Thus, if a field extension $F$ of $K$ is equivalent to $M$
as a structure of $\calL(\mbox{ring},K)$, then $F\models\mbox{Ax}(K,M)$.

Conversely, if $F\models\mbox{Ax}(K,M)$, then a monic polynomial
$f\in K[X]$ has a root in $F$ if and only if $f$ has a root
in $M$.
By [FrJ08, p.~441, Lemma 20.6.3(b)], $F\cap\tilde K\isom_K M$.
Hence, by [FrJ08, p.~436, Cor.~20.4.2], $F$ is elementarily
equivalent to $M$ as a structure of $\calL(\mbox{ring},K)$.

It follows from G\"odel completeness theorem
[FrJ08, p.~154, Cor.~8.2.6] that a
sentence $\theta$ of $\calL(\mbox{ring},K)$ is true in $M$ if and
only if $\theta$ has a formal proof in $\calL(\mbox{ring},K)$ from the
axioms $\mbox{Ax}(K,M)$ [FrJ08, Sec.~8.1, p.~150].
Thus, to check whether $\theta$ is true in $M$, one makes a
list of all formal proofs in $\calL(\mbox{ring},K)$ from the
axioms $\mbox{Ax}(K,M)$ and check them one by one.
After finitely many steps, one finds a proof of $\theta$ or of
$\neg\theta$.
In the first case $\theta$ is true in $M$ in the latter case
$\theta$ is false in $M$.
Consequently, $\mbox{Th}(M)$ is decidable.
\end{proof}

\begin{proposition}
 \label{DECc}
Let $K$ be a presented field with elimination theory.
Suppose that $K$ is Hilbertian.
Then, we can
for every positive integer $e$
recursively construct a decidable perfect PAC algebraic
extension $M$ of $K$
with $\Gal(M)\isom\hat F_e$ which is recursive in $\tilde K$.
\end{proposition}
\begin{proof}
We construct a primitive recursive list
$(G_1,G_2,G_3,\ldots)$
of all finite non-trivial groups that are generated by $e$
elements.
As mentioned in the first paragraph of this section,
$\tilde K$ has the splitting algorithm.
Hence, by [FrJ08, p.~405, Lemma 19.1.3(c)] applied to
$\tilde K$ rather than to $K$, we may find out whether a given
polynomial $f\in K[T,X]$ is irreducible over
$\tilde K$.
We use this test to
build a recursive list
$(f_1,f_2,f_3,\ldots)$ of all absolutely irreducible
polynomials in $K[T,X]$ that are monic and separable in $X$.
The rest of the proof breaks up into several parts.\\
{\bf Part A}:
{The induction plan}
By induction we effectively construct an ascending
sequence $(N_0,N_1,N_2,\ldots)$ of
presented finite Galois extensions of $K$ in $K_s$
and for each $n$ a presented
subfield $M_n$ of $N_n$ that contains $K$,
such that the following conditions hold for each $n\in\N$:
\be
\item[(1a)] \label{cf1a}
$z_n\in N_n$ (where, as in the beginning of this section,
$z_1,z_2,z_3,\ldots$ present $K_s$ over $K$).

\item[(1b)] \label{cf1b}
There exist $a\in K$ and $b\in M_n$ such that $f_n(a,b)=0$.

\item[(1c)] \label{cf1c}
The group $\Gal(N_n/M_n)$ is generated by $e$ elements,
it has $G_n$ as a quotient, and
$N_{n-1}\cap M_n=M_{n-1}$.
\ee
 {\bf Part B}:
{The field $N'_n$}
We start the induction by setting $M_0=N_0=K$.
Next we consider $n\ge1$ and assume that $N_0,\ldots,N_n$
and $M_0,\ldots,M_n$
have already been effectively constructed such that \eqref{cf1} holds
with $n$ replaced by $m$ for
$m=0,\ldots,n$.

Since $f_{n+1}$ is absolutely irreducible,
$f_{n+1}$ is irreducible over $N_n$.
Hence, we can
use Lemma \ref{DECb} to construct
a separable Hilbert subset $H$ of $K$ such that $f_{n+1}(a,X)$ is
irreducible over $N_n$ for each $a\in H$.
Then, we use Lemma \ref{DECa}
to effectively choose $a\in K$ such that $a\in H$.
In the next step we choose $b\in K_s$ with $f_{n+1}(a,b)=0$.
Hence, $N_n$ and $K(b)$ are linearly disjoint over $K$,
so $N_n\cap M_n(b)=M_n$.
Therefore, $\mbox{res}\colon\Gal(N_n(b)/M_n(b))\to\Gal(N_n/M_n)$
is an epimorphism.
We use [FrJ08, p.~412, Lemma 19.3.2] to effectively construct the
Galois closure $N'_n$ of $N_n(b,z_{n+1})/K$.
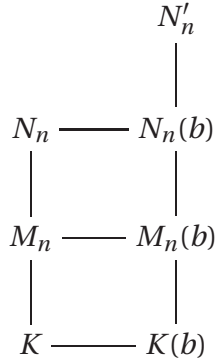
\[
\xymatrix{
& N'_n \ar@{-}[d]
\cr
N_n \ar@{-}[r] \ar@{-}[d]
& N_n(b) \ar@{-}[d]
\cr
M_n \ar@{-}[r] \ar@{-}[d]
& M_n(b) \ar@{-}[d]
\cr
K \ar@{-}[r]
& K(b)
}
\]
{\bf Part C}:
{Construction of $N_{n+1}$}.
We compute the order $r$ of $G_{n+1}$ and
embed $G_{n+1}$ into the symmetric group $S_r$.
For every field $F$, the Galois group of the general
polynomial $X^r+T_1X^{r-1}+\cdots+T_r$ over
$F(T_1,\ldots,T_r)$ is the symmetric group $S_r$
[Lan93, p.~272, Example VI.2.2].
The proof of
[FrJ08, p.~231, Lemma 13.3.1] gives
a separable Hilbert subset $H$
of $F^r$ such that 
$\Gal(X^r+a_1X^{r-1}+\cdots+a_r,F)\isom S_r$
for each $\mathbf a\in F^r$.

Next we compute the number $s$ of subfields
of $N'_n$ that properly contain $K$,
and use the preceding paragraph
and Lemmas \ref{DECa} and \ref{DECb}
to construct
$s+1$ linearly disjoint Galois extensions
$L_1,\ldots,L_{s+1}$ of $K$ with Galois group $S_r$.
The intersection of at least one of these fields with
$N'_n$ is $K$.
Computing the intersections
$N'_n\cap L_1,\ldots,N'_n\cap L_{s+1}$, we find an $i$
between $1$ and $s+1$ such that $N'_n\cap L_i=K$.
In other words, $N'_n$ and $L_i$ are linearly disjoint over $K$.
Hence, $N'_n$ and $M_n(b)L_i$ are linearly disjoint over $M_n(b)$.
We set $N_{n+1}=N'_nL_i$.\\
{\bf Part D}:
{Construction of $M_{n+1}$}. \setcounter{equation}{1}
By the preceding paragraph,
\begin{equation}
\label{cf2}
\Gal(N_{n+1}/M_n(b))
\isom\Gal(N'_n/M_n(b))\times\Gal(M_n(b)L_i/M_n(b)).
\end{equation}
In addition, $M_n(b)$ is linearly disjoint from $L_i$ over
$K$.
We effectively find $\tau_1,\ldots,\tau_e$ in
$\Gal(L_i/K)$ that
generate a subgroup which is isomorphic to $G_{n+1}$
[FrJ08, p.~412, Lemma 19.3.2]
and set $K_i$ to be the fixed field of
$\tau_1,\ldots,\tau_e$ in $L_i$.
Then, $\Gal(M_n(b)L_i/M_n(b)K_i)\isom\Gal(L_i/K_i)\isom
G_{n+1}$.
$$
\xymatrix{
& N'_n \ar@{-}[rr] \ar@{-}[d]
&& N_{n+1} \ar@{-}[dd]
\\
N_n \ar@{-}[r] \ar@{-}[d]
& N_n(b) \ar@{-}[d]
\\
M_n \ar@{-}[r] \ar@{-}[dd]
& M_n(b) \ar@{-}[r] \ar@{-}[d]
& M_n(b)K_i \ar@{-}[r]^{G_{n+1}} \ar@{-}[d]
& M_n(b)L_i \ar@{-}[d]
\\
& K(b) \ar@{-}[r] \ar@{-}[dl]
& K_i(b) \ar@{-}[r]^{G_{n+1}} \ar@{-}[d]
& L_i(b) \ar@{-}[d]
\\
K \ar@{-}[rr]
&& K_i \ar@{-}[r]^{G_{n+1}}
& L_i
}
$$

By (1c), we find
$\sigma_{n,1},\ldots,\sigma_{n,e}$ of
$\Gal(N'_n/M_n(b))$ whose restriction to $N_n(b)$
generate $\Gal(N_n(b)/M_n(b))$, so their restrictions to $N_n$
generate $\Gal(N_n/M_n)$.
By \eqref{cf2}, we can effectively find
$\sigma_{n+1,1},\ldots,\sigma_{n+1,e}$
in $\Gal(N_{n+1}/M_n(b))$
whose restrictions to $N'_n$ are
$\sigma_{n,1},\ldots,\sigma_{n,e}$, respectively,
and whose restrictions to $L_i$ are $\tau_1,\ldots,\tau_e$,
respectively.
Thus, by (1c), both $G_n$ and $G_{n+1}$ are quotients of the subgroup
$H=\langle\sigma_{n+1,1},\ldots,\sigma_{n+1,e}\rangle$
of $\Gal(N_{n+1}/M_n(b))$.
Then, the restriction of $H$ to $L_i$ is $G_{n+1}$ and the
restriction of $H$ to $N_n$ is $\Gal(N_n/M_n)$.
Let $M_{n+1}$ be the fixed field of $H$ in $N_{n+1}$.
Then, $G_{n+1}$ is a quotient of $\Gal(N_{n+1}/M_{n+1})$,
and $N_n\cap M_{n+1}=M_n$.
This concludes the $(n+1)$th step of the induction.

We put all of the fields
mentioned above appears in the following
diagram of fields:
$$
\xymatrix{
&& N'_n \ar@{-}[r] \ar@{-}[dl]
& N_{n+1} \ar@{-}[dl]^H \ar@{-}[dd] \ar@{-}[dll]
\\
N_n \ar@{-}[r] \ar@{-}[d]
& N_n(b) \ar@{-}[d]
& M_{n+1} \ar@{-}[d] \ar@{-}[dl]
\\
M_n \ar@{-}[r] \ar@{-}[d]
&
M_n(b) \ar@{-}[r] \ar@{-}[d]
& M_n(b)K_i \ar@{-}[r]^{G_{n+1}} \ar@{-}[d]
& M_n(b)L_i \ar@{-}[d]
\\
K \ar@{=}[r]
& K \ar@{-}[r]
& K_i \ar@{-}[r]^{G_{n+1}}
& L_i
}
$$
{\bf Part E}:
{The field $M_\infty$}
By the defining property of $z_1,z_2,z_3,\ldots$
and by (1a),
$\union_{n=1}^\infty N_n=K_s$.
By Part A, $M_\infty=\union_{n=1}^\infty M_n$ is a presented
recursive subfield of $K_s$.
Moreover, for $n'>n$, (1c) and induction on $n'-n$
imply that $M_{n'}\cap N_n=M_n$.,
Hence, $M_\infty\cap N_n=M_n$ for each positive integer $n$.
Also,
$\Gal(M_\infty)$ is the inverse limit of the groups
$\Gal(N_n/M_n)$.
Since each of these groups is generated by $e$ elements, so is
$\Gal(M_\infty)$ (as a profinite group).
In addition, since $G_n$ is a quotient of $\Gal(N_n/M_n)$,
each finite
group which is generated by $e$ elements is a quotient of
$\Gal(M_\infty)$.
Hence, $\Gal(M_\infty)\isom\hat F_e$ [FrJ08, p.~360, Lemma 17.7.1].
Finally, by (1b), each absolutely irreducible polynomial
in two variables with coefficients in $K$ has a zero in
$M_\infty$.
Therefore, by [FrJ08, p.~195, Thm.~11.2.3], $M_\infty$ is PAC.\\
{\bf Part F}:{The field $M_\infty$ is recursive in $K_s$}
Next we show how to decide whether a given monic separable
polynomial $f\in K[X]$ has a root in $M_\infty$.
Since $K$ has the splitting algorithm, we may assume that $f$
is irreducible.
Moreover, since $K_s$ has the splitting algorithm, we may
find a root $z$ of $f$ in $K_s$ and identify $z$ as $z_n$ for
some positive integer $n$.
By \eqref{cf1a}, $z\in N_n$ and so $f$ totally splits in $N_n$.
We check whether $G_n$ fixes any of the roots of $f$
(by [FrJ08, p.~412, Lemma 19.3.2]).
This will be the case if and only if $f$ has a root in
$M_n$.
Since, by Part E, $N_n\cap M_\infty=M_n$, this will be the case if and
only if $f$ has a root in $M_\infty$.

In the situation of the preceding paragraph, we may check
whether $G_n$ fixes $z$, hence whether $z\in M_\infty$.
This proves that $M_\infty$ is recursive in
$K_s$.\\
{\bf Part G}: {\it Conclusion of the proof}.
Finally, let $M$ be the maximal purely inseparable extension
of $M_\infty$ in $\tilde K$.
Then, $M$ is recursive in $\tilde K$ (Lemma \ref{INS}),
$M$ is PAC [FrJ08, p.~196, Cor.~11.2.5], and
$\Gal(M)\isom\Gal(M_\infty)\isom\\hat F_e$.
It follows from Lemma \ref{JAX} that $M$ is decidable.
\end{proof}

We are now in a position to prove a stronger version of
Theorem B of the introduction.

\begin{theorem}
\label{DECd}
Let $K$ be a presented field with elimination theory.
Suppose that $K$ is Hilbertian.
Then, we can
for every positive integer $e$
construct an infinite sequence of decidable PAC
perfect fields which are recursive in $\tilde K$, each with absolute Galois
group isomorphic to $\\hat F_e$.
\end{theorem}

\begin{proof}
Let $n$ be a non-negative number and assume that we have
already constructed $n$ distinct decidable PAC perfect fields
$M^{(1)},\ldots,M^{(n)}$ with absolute Galois groups
isomorphic to $\\hat F_e$ and which are recursive in $\tilde K$.
In particular, $M^{(1)},\ldots,M^{(n)}$ are proper
$K$-vector-subspaces of $\tilde K$.
Hence, $\bigcup_{i=1}^nM^{(i)}$ is a proper subset of $\tilde K$
[Hup90, p.~11, A 1.1.c].
Since each $M^{(i)}$ is a recursive subset of $\tilde K$, we
may find $z\in\tilde K\setminus\bigcup_{i=1}^nM^{(i)}$.

Since $K$ has elimination theory, so has $K(z)$ [FrJ08,
p.~410, Def.~19.2.8].
By [FrJ08, p.~224, Cor.~12.2.3],
$K(z)$ is Hilbertian.
Hence, by Proposition \ref{DECc}, $K(z)$ has an
extension $M^{(n+1)}$ which is perfect, PAC, decidable,
$\Gal(M^{(n+1)})\isom\hat F_e$, and is a recursive subfield
of $\tilde K$.
Since $z\in M^{(n+1)}$, we have $M^{(n+1)}\ne M^{(i)}$ for
$i=1,\ldots,n$.
This concludes the induction.
\end{proof}

\begin{remark} \label{DECe}
If $N$ is a Galois extension of $K$ and $\mbox{Ex}(N)$ is
decidable (resp.~primitively decidable), then $N$ is also a
recursive (resp.~primitive recursive) subfield of $K_s$,
hence also of $\tilde K$.
Indeed, if $z\in K_s$, we construct $\mbox{irr}(z,K)$
and check whether $\mbox{irr}(z,K)$ has a root in $N$.
This is the case if and only if all of the roots of
$\mbox{irr}(z,K)$ belong to $N$,
hence if and only if $z\in N$.

For example, the field $\Q_{\rm tr}$ of all totally real
algebraic numbers is a Galois extension of $\Q$ and
$\mbox{Th}(\Q_{\mbox{tr}})$ is primitively decidable.
[FHV94, p.~90, Thm.~10.1].
If $S$ is a finite number of prime numbers,
then the field $\Q_{\mbox{totS}}$ of all totally $S$-adic numbers is
decidable [Ers96].
By the preceding paragraph, $\Q_{\mbox{tr}}$ is primitive
recursive in $\tilde \Q$ and $\Q_{\mbox{totS}}$ is  
recursive in $\tilde \Q$.
\end{remark}

\begin{remark}
\label{DECf}
{Primitive recursive procedures}
All of the procedures we use to construct the field $M$ in
Proposition \ref{DECc} are primitive recursive except
the procedure we use in Lemma \ref{DECa}.
Thus, if we can effectively find a point in every separable Hilbert
subset of $K$ (in which case we say that is {\bf $K$ 
effectively Hilbertian}), then we can effectively construct $M$
in Proposition \ref{DECc} and then also effectively construct the
infinite sequence of fields $M^{(1)},M^{(2)},M^{(3)},\ldots$
that appears in Theorem \ref{DECd}.
In addition, we can prove that the theories of those fields
are primitive recursive.

We can prove the above mentioned effectiveness results at
least in the case where $K$ is an infinite finitely
generated field (over its prime field).
Unfortunately, these improvements would need a lot of space
and so would go beyond the scope of this work.
So, we restrict ourselves to some hints for the
proofs.

In order to prove that a presented finitely generated
infinite field $K$ is effectively Hilbertian, it suffices to
prove it only in the cases where $K$ is either $K_0(t)$, where $K_0$ is
an infinite finitely generated field and $t$ is
indeterminate, or $K=\Q$, or $K=\F_p(t)$. 

In the first case, we notice that the proof of [FrJ08,
p.~236, Prop.~13.2.1] reduces the effective Hilbertianity to
the primitive recursiveness of $\mbox{Th}(\tilde K_0)$, which is
proved in [FrJ08, p.~170, Prop.~9.4.3].

The case where $K=\Q$ or $K=\F_p(t)$ involves effective
operations in $\Z$ like factoring positive integers into
products of prime numbers and effective Chebotarev density
theorem.
All of this appear in the proof of Theorem 13.3.5 of
[FrJ08] and the lemmata that proceed it in Chapter 13 of
[FrJ08].

Finally, we have to improve Lemma \ref{JAX} and prove
that, under the conditions of that lemma, $\mbox{Th}(M)$ is
primitive recursive.
This depends on ``elimination of quantifiers in the language
of Galois stratification'' given by [FrJ08, p.~721,
Prop.~30.5.3].
\end{remark}

%End of file decide.tex
%File appendix.tex, 3 December 2014
%6 December 2014
%6 January 2015
\section{Appendix}
\label{APP}

We prove in this appendix a statement made in the
introduction.

For each field $K$, every positive integer, and every
${\boldsymbol \sigma}\in\Gal(K)^e$ we denote the maximal purely
inseparable extension of $K_s({\boldsymbol \sigma})$ by $\tilde K({\boldsymbol \sigma})$.

\begin{proposition}
\label{APPa}
Let $K$ be a Hilbertian field and let $e$ be a positive integer.
Then, there are uncountably many elementary equivalence
classes (in the language $\calL(\mbox{ring},K)$) of fields of the
form $\tilde K({\boldsymbol \sigma})$ with ${\boldsymbol \sigma}\in\Gal(K)^e$.

Moreover, the Haar measure of the set of pairs
$({\boldsymbol \sigma},{\boldsymbol \sigma}')\in\Gal(K)^{2e}$ such that
$K_s({\boldsymbol \sigma})$ and $K_s({\boldsymbol \sigma}')$ are not equivalent as
structures of the language $\calL(\mbox{ring},K)$ is $1$.
\end{proposition}

\begin{proof}
Using the assumption that $K$ is Hilbertian,
we construct, by induction, a linearly disjoint sequence
$K_1,K_2,K_3,\ldots$ of quadratic separable extensions of
$K$ [FrJ08, p.~297, Cor.~16.2.7(b)] (if $\mbox{chr}(K)=2$,
one has to use [FrJ08, p.~296, Example 16.2.5(c) and p.~297,
Lemma 16.2.6]).
Let $L$ be the compositum of all these extensions.
Then, $\Gal(L/K)$ is an infinite profinite group of exponent
$2$.
In particular, the closed subgroup generated by every finite
subset is finite, hence has Haar measure $0$
in $\Gal(L/K)$.

We denote the normalized Haar measure of $\Gal(L/K)$ by $\mu$.
For each\break
 ${\boldsymbol \sigma}=(\sigma_1,\ldots,\sigma_e)\in\Gal(K)^e$
we denote the fixed
field of $\sigma_1,\ldots,\sigma_e$ in $L$ by $L({\boldsymbol \sigma})$ and
let $\langle{\boldsymbol \sigma}\rangle=\Gal(L/L({\boldsymbol \sigma}))$ be the closed
subgroup of $\Gal(L/K)$ generated by $\sigma_1,\ldots,\sigma_e$.

If there are only countably many fields
$L({\boldsymbol \sigma}^{(1)}),L({\boldsymbol \sigma}^{(2)}),L({\boldsymbol \sigma}^{(3)}),\ldots$
with ${\boldsymbol \sigma}^{(i)}\in\Gal(L/K)^e$,
then
$\Gal(L/K)=\bigcup_{i=1}^\infty\langle{\boldsymbol \sigma}^{(i)}\rangle$,
so
$
1=\mu(\Gal(L/K))
\le\sum_{i=1}^\infty\mu(\langle{\boldsymbol \sigma}^{(i)}\rangle)
=0,
$
which is a contradiction.
Hence, there is an uncountable subset $S$ of $\Gal(L/K)^e$
such that $L({\boldsymbol \sigma})\ne L({\boldsymbol \sigma}')$ for every distinct
elements ${\boldsymbol \sigma}$ and ${\boldsymbol \sigma}'$ of $S$.

We extend each ${\boldsymbol \sigma}\in S$ to an element $\tilde{\boldsymbol \sigma}$ of
$\Gal(K)^e$.
If ${\boldsymbol \sigma}'\in S$ and ${\boldsymbol \sigma}'\ne{\boldsymbol \sigma}$, then
$K_s(\tilde{\boldsymbol \sigma})$ is not $K$-conjugate to
$K_s(\widetilde{{\boldsymbol \sigma}'})$,
otherwise $L({\boldsymbol \sigma})$ and $L({\boldsymbol \sigma}')$ are $K$-conjugate,
hence equal, because $\Gal(L/K)$ is abelian.
It follows from [FrJ08, p.~441, Lemma 20.6.3(b)] that
$\tilde K(\tilde{\boldsymbol \sigma})$ and $\tilde K(\widetilde{{\boldsymbol \sigma}'})$ are not
elementarily equivalent as structures of $\calL(\mbox{ring},K)$.
This proves the first statement of the proposition.

The proof of the second statement of the proposition is
based on the observation that the diagonal $D$ of
$\Gal(L/K)^e$ has Haar measure $0$ in $\Gal(L/K)^{2e}$.
This is so, because $\Gal(L/K)^e$ is an infinite profinite
group and for every finite group $G$, the proportion of the
diagonal $\{(g,g)\in G^2| g\in G\}$ in $G^2$ is $1\over|G|$.
It follows from [Hal68, p.~279, Thm.~C] that the set
$\tilde D'=\{({\boldsymbol \sigma},{\boldsymbol \sigma}')\in\Gal(K)^{2e}| {\boldsymbol \sigma}|_L\ne{\boldsymbol \sigma}'|_L\}$
has Haar measure $1$.
By the preceding paragraph, $K_s({\boldsymbol \sigma})$ is not
elementarily equivalent to $K_s({\boldsymbol \sigma}')$ for all
$({\boldsymbol \sigma},{\boldsymbol \sigma}')\in\tilde D$.
\end{proof}
%End of file appendix.tex
%File ref.tex, 28 November 2014
%6 January 2015

\end{document}